\documentclass{ifacconf}

\usepackage{graphicx}      % include this line if your document contains figures
\usepackage{natbib}        % required for bibliography
\newtheorem{theorem}{Theorem}[section]
\newtheorem{corollary}[theorem]{Corollary}
\newtheorem{lemma}[theorem]{Lemma}
\newtheorem{proposition}[theorem]{Proposition}

\theoremstyle{definition}
\newtheorem{definition}[theorem]{Definition}
\newtheorem{remark}{Remark}
\begin{document}
\begin{frontmatter}

\title{Asymptotic Solution of a Cheap Control Game with Slow and Fast State Variables}  %%%%Title\thanksref{footnoteinfo}}

%\thanks[footnoteinfo]{Sponsor and financial support acknowledgment
%goes here. Paper titles should be written in uppercase and lowercase
%letters, not all uppercase.}

\author[First]{Valery Y. Glizer}\ {\bf and}\
\author[Second]{Vladimir Turetsky}

\address[First]{Department of Mathematics, Braude College of Engineering,
   Karmiel, 2161002 Israel (e-mail: valgl120@gmail.com).}
\address[Second]{Department of Mathematics, Braude College of Engineering,
   Karmiel, 2161002 Israel (e-mail: turetsky1@braude.ac.il)}

\begin{abstract}                % Abstract of 50--100 words
A finite-horizon zero-sum linear-quadratic differential game is considered. Its features are: (i) the control cost of the minimizing player in the game's cost functional is much smaller than the control cost of the maximizing player and the state cost; (ii) the cost of the fast state variable in the integrand of the cost functional is a positive semi-definite (but non-zero) quadratic form. These features require developing a significantly novel approach to asymptotic analysis of the matrix Riccati differential equation associated with the considered game. Using this analysis, an asymptotic solution of the game is derived. An illustrative example is presented.
\end{abstract}


\end{frontmatter}
%===============================================================================

\section{Introduction}
We study a two-person finite-horizon zero-sum linear-quadratic differential game. The features of this game are the following: (i) the control cost of the minimizer (the minimizing player) in the game's cost functional is small in comparison with the control cost of the maximizer (the maximizing player) and  with the cost of the state variables; (ii) the matrix-valued coefficient of the cost of the fast state variable in the integral part of the cost functional is a positive semi-definite (but non-zero) matrix. The feature (i) means that the considered game is a cheap control game. Cheap control problems have a considerable importance in such topics as: (1) singular controls and arcs (see, e.g, \cite{Glizer-Kelis-book} and references therein);
(2) limiting forms and maximally achievable accuracy of optimal regulators and filters (see, e.g., \cite{Kwakernaak-Sivan-1972,Braslavsky-et-al-1999} and references therein); (3) inverse optimal control problems (see e.g. \cite{Moylan-Anderson-1973}); (4) high gain control problems (see, e.g., \cite{Kokotovic-Khalil-OReilly-1999} and references therein).
Due to the smallness of the control cost, the Hamilton boundary-value problem and the Hamilton-Jacobi-Bellman-Isaacs
equation, associated with a cheap control problem by solvability conditions, are singularly perturbed. This feature means that cheap control problems are sources of novel classes of singularly perturbed differential equations, thus having a considerable importance in the theory of differential equations.\\
As it is aforementioned, here we study a cheap control differential game. Cheap control differential games are extensively investigated in the literature. The review of this topic can be found in \cite{Glizer-Turetsky-Symmetry}. In most of the works, devoted to this topic, two types of the quadratic cost of the fast state variable in the integral part of the cost functional are considered: (a) the cost is a positive definite quadratic form; (b) the cost is zero.
In this paper, we study an intermediate case, i.e., the case where the quadratic cost of the fast state variable in the integral part of the cost functional is a positive semi-definite (but non-zero) quadratic form.
More precisely, in this paper, we consider the finite-horizon zero-sum linear-quadratic differential game with the cheap control of the minimizer. The game's dynamics has both, slow and fast, state variables. The cost of the fast variable in the integrand of the game's cost functional is a positive semi-definite (but non-zero) quadratic form. This game is an essential generalization of the game studied in \cite{Glizer-Turetsky-Symmetry} where the entire state variable is fast, and the matrix-valued coefficient for the minimizer's control in the dynamics is quadratic and invertible.\\
State-feedback saddle-point solution of the game, considered in the present paper, is reduced to the solution of the singularly perturbed terminal-value problem for the game-theoretic matrix Riccati differential equation. Asymptotic solution of this problem is derived. The derivation of this asymptotic solution requires the developing an essentially novel method, which is a significant generalization of the method of asymptotic solution of a singularly perturbed matrix Riccati differential equation known in the literature (see, e.g., \cite{Glizer-Kelis-book,Kokotovic-IEEE} and references therein).\\
Based on the aforementioned asymptotic solution to the terminal-value problem for the game-theoretic Riccati equation, an asymptotic approximation of the game's value and an approximate-saddle point are derived.\\
The main notations, applied in the paper, are: (I) $E^{n}$ is the $n$-dimensional real Euclidean space; (II) $\|\cdot\|$ is the Euclidean norm either of a vector, or of a matrix; (III) upper index $^T$ denotes the transposition either of a vector, or of a matrix; (IV) $I_{n}$ is the identity matrix of dimension $n$; (V) ${\rm col}(x,y)$, where $x \in E^{n}$, $y \in E^{m}$, is a column block vector with the upper block $x$ and the lower block $y$; (VI) ${\rm diag}(a_{1},...,a_{n})$ is a diagonal matrix with the diagonal entries $a_{1}$,... $a_{n}$; (VII) $L^{2}[t_{1},t_{2}; E^{n}]$ is the space of all functions $z(\cdot): [t_1,t_2] \rightarrow E^{n}$ square integrable in the interval $[t_1,t_2]$.

\section{Problem Statement}\label{probl-stat}

\subsection{Game Formulation\label{Sec2.1}}

We consider a two-player differential game which dynamics is described by the system
\begin{equation}\label{eq-x}\begin{array}{c}
dx(t)/dt=A_{1}(t)x(t)+A_{2}(t)y(t)\\
+C_{1}(t)v(t),\  \ t \in [0,t_{f}],\  \ x(0)=x_{0},\end{array}
\end{equation}
\begin{equation}\label{eq-y}\begin{array}{c}
dy(t)/dt=A_{3}(t)x(t)+A_{4}(t)y(t)\\ +u(t)+C_{2}(t)v(t),\  \ t \in [0,t_{f}],\  \ y(0)=y_{0},\end{array}
\end{equation}
where $t_{f} > 0$ is a given time instant; $x(t)\in E^{n}$ and $y(t)\in E^{m}$, $(m > 1)$ are the state vectors; $u(t)\in E^{m}$ and $v(t)\in E^{l}$ are the players' controls; $A_{i}(t)$, $(i=1,..,4)$ and
$C_{j}(t)$, $(j=1,2)$ are given matrices of corresponding
dimensions; $x_{0}\in E^{n}$ and $y_{0}\in E^{m}$ are given vectors; the matrix-valued functions $A_{i}(t)$, $(i=1,..,4)$ and $C_{j}(t)$, $(j=1,2)$ are continuous in the interval $[0,t_{f}]$.

The cost functional, to be minimized by the control $u$ (the minimizer's control) and maximized by the control $v$ (the maximizer's control), is
\begin{equation}\begin{array}{c}
J(u,v)= x^{T}(t_{f})F_{1}x(t_{f})
+ \int\limits_{0}^{t_{f}}\big[x^{T}(t)D_{1}(t)x(t)\\ + y^{T}(t)D_{2}(t)y(t) + \varepsilon^{2}u^{T}(t)u(t)
-v^{T}(t)G(t)v(t)\big]dt,\end{array}
\label{perf-ind}
\end{equation}
where $F_{1}$, $D_{j}(t)$, ($j=1,2$), $G(t)$ are given symmetric matrices of corresponding dimensions; $F_{1}$ is constant positive semi-definite; for any $t \in [0,t_{f}]$, $D_{1}(t)$ and $D_{2}(t) \ne 0$ are positive semi-definite, while $G(t)$ is positive definite; the matrix-valued functions $D_{j}(t)$, ($j=1,2$) and $G(t)$ are continuous in the interval $[0,t_{f}]$; $\varepsilon > 0$ is a small parameter.
\begin{remark}\label{cheap-contr-game}We assume that both players of the game (\ref{eq-x})-(\ref{perf-ind}) are aware of all the game's data and of the game's current position $\{x(t),y(t),t\}$. Since the parameter $\varepsilon > 0$ is small, the control cost of the minimizer is small in comparison with the state cost in the functional (\ref{perf-ind}). Hence, the game (\ref{eq-x})-(\ref{perf-ind}) is a cheap control game. We call the game (\ref{eq-x})-(\ref{perf-ind}) the Cheap Control Differential Game (CCDG).
\end{remark}
\begin{remark}\label{slow-fast-state-variables} The nonsingular control transformation $\widehat{u}(t) = \varepsilon u(t)$, ($\widehat{u}(t)$ is a new minimizer's control) converts the CCDG into the equivalent zero-sum differential game consisting of the differential equations (\ref{eq-x}) and
\begin{equation}\label{sing-pert-eq-y}\begin{array}{c} \varepsilon\frac{dy(t)}{dt} = \varepsilon\big[A_{3}(t)x(t) + A_{4}(t)y(t)\\ + C_{2}(t)v(t)\big] + \widehat{u}(t),\  \ t \in [0,t_{f}],\  \ y(0) = y_{0},\end{array}
\end{equation}
and of the cost functional
\begin{equation}\label{no-cheap-contr-funct}\begin{array}{c}\widehat{J}(\widehat{u},v) = x^{T}(t_{f})F_{1}x(t_{f}) +  \int_{0}^{t_{f}}\big[x^{T}(t)D_{1}(t)x(t)\\
+ y^{T}(t)D_{2}(t)y(t) + \widehat{u}^{T}(t)\widehat{u}(t) - v^{T}(t)G(t)v(t)\big]dt.\end{array}
\end{equation}
The dynamic system of the new game is singularly perturbed where $x(t)$ is a slow state variable and $y(t)$ is a fast state variable (see, e.g., \cite{Vasil'eva-SIAM}). Therefore, we call the state variables $x(t)$ and $y(t)$ of the CCDG a slow state variable and a fast state variable, respectively.
\end{remark}
\begin{remark}\label{D_2-generalization} The assumption that $D_{2}(t)$ is positive semi-definite is a considerable generalization of the results of \cite{Glizer-Kelis-book}. In this book, the asymptotic analysis of the cheap control  zero-sum linear-quadratic differential game is essentially based on the assumption that the matrix $D_{2}(t)$ is positive definite for all $t \in [0,t_{f}]$. In \cite{Glizer-Turetsky-Symmetry}, a particular case of the game (\ref{eq-x})-(\ref{perf-ind}) is analyzed. Namely, in \cite{Glizer-Turetsky-Symmetry}, the slow state variable $x(t)$ is absent, i.e., the dynamics consists only of the equation (\ref{eq-y}) with $A_{3}(t) \equiv 0$ and in the cost functional $F_{1} = 0$, $D_{1}(t) \equiv 0$, $D_{2}(t) \ne 0$ is positive semi-definite, $G(t)$ is positive definite.
\end{remark}
\begin{remark}\label{D_2-form} For the sake of the simplicity, we assume that the matrix $D_{2}(t)$ has the form: 
\begin{equation}\label{form-D_2}\begin{array}{c} D_{2}(t) = {\rm diag}\big(\lambda_{1}(t),\lambda_{2}(t),...,\lambda_{m}(t)\big),\\ 
\lambda_{k}(t) \ge 0,\  \  \ k = 1,...,m,\  \ t \in [0,t_{f}].\end{array}
\end{equation}
\end{remark}

\subsection{Main Definitions \label{Subsec2.2}}
Let us consider the block vectors
\begin{equation}\label{z-z_0}z \stackrel{\triangle}{=} {\rm col}(x , y) \in E^{n+m},\  \  \  \ z_{0}\stackrel{\triangle}{=}{\rm col}(x_{0} , y_{0}) \in E^{n+m}.\end{equation}
By ${\mathcal{U}}$, we denote the set of all functions $u=u(z,t):
E^{n+m}\times [0,t_{f}] \rightarrow E^{m}$, which are measurable w.r.t. $t \in [0,t_{f}]$ for any fixed $z \in E^{n+m}$ and satisfy the local Lipschitz condition w.r.t. $z \in E^{n+m}$ uniformly in $t \in [0,t_{f}]$. Similarly, by ${\mathcal{V}}$, we denote the set of all functions $v=v(z,t): E^{n+m}\times [0,t_{f}] \rightarrow E^{l}$, which are measurable w.r.t. $t \in [0,t_{f}]$ for any fixed $z \in E^{n+m}$ and satisfy the local Lipschitz condition w.r.t. $z \in E^{n+m}$ uniformly in $t \in [0,t_{f}]$.

The following definitions are based on the results of \cite{Glizer-Kelis-book}.
\begin{definition}\label{admis-pair}Let $UV$ be the set of all pairs\\
$\big(u(z,t),v(z,t)\big)$ such that: (i) $u(z,t)\in {\mathcal{U}}$, $v(z,t)\in {\mathcal{V}}$;
(ii) the initial-value problem (\ref{eq-x})-(\ref{eq-y}) for $u(t)=u(z,t)$,
$v(t)=v(z,t)$ and any $x_{0}\in E^{n}$, $y_{0}\in E^{m}$ has the unique absolutely continuous solution $z_{uv}(t; z_{0})={\rm col}\big(x_{uv}(t; z_{0}) , y_{uv}(t; z_{0})\big)$, $\big(x_{uv}(t; z_{0})\in E^{n}\ ,\ y_{uv}(t; z_{0})\in E^{m}\big)$ in the entire interval $[0,t_{f}]$; (iii) $u\big(z_{uv}(t; z_{0}),t\big)\in
L^{2}[0,t_{f}; E^{m}]$; (v) $v\big(z_{uv}(t; z_{0}),t\big)\in L^{2}[0,t_{f}; E^{l}]$.
In what follows, $UV$ is called the set of all admissible pairs of the players' state-feedback
controls in the CCDG.
\end{definition}
For a given $u(z,t)\in {\mathcal U}$, consider the sets
${\mathcal E}_{v}\big(u(z,t)\big)\stackrel{\triangle}{=}\{v(z,t)\in{\mathcal V}:\big(u(z,t),v(z,t)\big)\in UV\}$
and
${\mathcal H}_{u}\stackrel{\triangle}{=}\{u(z,t)\in{\mathcal U}: {\mathcal E}_{v}\big(u(z,t)\big)\ne\emptyset\}$.\\
Similarly, for a given $v(z,t)\in {\mathcal V}$, consider the sets
${\mathcal E}_{u}\big(v(z,t)\big)\stackrel{\triangle}{=}\{u(z,t)\in{\mathcal U}:\big(u(z,t),v(z,t)\big)\in UV\}$
and
${\mathcal H}_{v}\stackrel{\triangle}{=}\{v(z,t)\in{\mathcal V}: {\mathcal E}_{u}\big(v(z,t)\big)\ne\emptyset\}.$
\begin{definition}\label{guaran-res-u}For a given $u(z,t)\in{\mathcal H}_{u}$, the value
\begin{equation}
J_{u}\big(u(z,t);z_{0}\big) =\sup_{v(z,t)\in {\mathcal E}_{v}\big(u(z,t)\big)}J\big( u(z,t),v(z,t)\big)  \label{guarant-res-u}\end{equation}
is called the guaranteed result of $u(z,t)$ in the CCDG.
\end{definition}
\begin{definition}\label{guaran-res-v}For a given $v(z,t)\in{\mathcal H}_{v}$, the value
\begin{equation}
J_{v}\big(v(z,t);z_{0}\big) =\inf_{u(z,t)\in {\mathcal E}_{u}\big(v(z,t)\big)}J\big(u(z,t),v(z,t)\big)  \label{guarant-res-v}\end{equation}
is called the guaranteed result of $v(z,t)$ in the CCDG.
\end{definition}
\begin{definition}\label{equilibr-SDG} A pair $\big(u^{*}(z,t) , v^{*}(z,t)\big)$ is called a saddle point of the CCDG if
\begin{equation}\label{equality-guar-res}J_{u}\big(u^{*}(z,t);z_{0}\big) = J_{v}\big(v^{*}(z,t);z_{0}\big)\  \ \forall z_{0} \in E^{n+m}.\end{equation}
In this case, the value 
\begin{equation}\label{game-value} J^{*}(z_{0}) = J_{u}\big(u^{*}(z,t);z_{0}\big) = J_{v}\big(v^{*}(z,t);z_{0}\big)\end{equation} 
is called a value of the CCDG.
\end{definition}

\subsection{Objectives of the Paper}
The objectives of this paper are: (I) to establish the existence of a saddle point and a value of the CCDG for all sufficiently small $\varepsilon > 0$;
(II) to derive asymptotic approximations for these saddle point and value.

\section{Solvability Condition of the CCDG}

Let us introduce the following block-form matrices:
\begin{eqnarray}\label{new-matrices}A(t) \stackrel{\triangle}{=} \left[\begin{array}{l}A_{1}(t)\  \ A_{2}(t)\\
A_{3}(t)\  \ A_{4}(t)\end{array}\right],\
B \stackrel{\triangle}{=} \left[\begin{array}{l}0\\ I_{m}\end{array}\right],\  C(t) \stackrel{\triangle}{=} \left[\begin{array}{l}C_{1}(t)\\ C_{2}(t)\end{array}\right],\nonumber\\
F \stackrel{\triangle}{=} \left[\begin{array}{l}F_{1}\  \ 0\\ 0\  \  \  \ 0\end{array}\right],\  \ D(t) \stackrel{\triangle}{=} \left[\begin{array}{l}D_{1}(t)\  \  0\\ 0\  \  \  \  \  \  \  \ D_{2}(t)\end{array}\right],\nonumber\\
\end{eqnarray}
where $A(t)$, $F$, $D(t)$ are of the dimension $(n+m)\times(n+m)$; $B$ is of the dimension $(n+m)\times m$; $C(t)$ is of the dimension $(n+m)\times l$.\\
Consider the terminal-value problem for the $(n+m)\times(n+m)$ matrix-valued function $K(t)$ in the interval $[0,t_{f}]$
\begin{equation}\label{eq-K}\begin{array}{c}dK(t)/dt = - K(t)A(t) - A^{T}(t)K(t)\\ + K(t)[S_{u}(\varepsilon) -  S_{v}(t)]K(t) - D(t),\ K(t_{f}) = F,\end{array}
\end{equation}
\begin{equation}\label{S_u-S_v}\begin{array}{l}S_{u}(\varepsilon) = \varepsilon^{-2}BB^{T} = \varepsilon^{-2}\left[\begin{array}{l}0\  \ 0\\ 0\  \ I_{m}\end{array}\right],\\
S_{v}(t) = C(t)G^{-1}(t)C^{T}(t).\end{array}
\end{equation}
We assume that:\\
{\bf A1.} For a given $\varepsilon > 0$, the terminal-value problem (\ref{eq-K}) has the symmetric solution $K(t) = K(t,\varepsilon)$ in the entire interval $[0,t_{f}]$.

Since the right-hand side of the differential equation in (\ref{eq-K}) is a smooth function with respect to $K(t)$, the solution $K(t) = K(t,\varepsilon)$ is unique.\\
For all $(z,t) \in E^{n+m}\times[0,t_{f}]$, consider the functions
\begin{equation}\label{u*-v*}\begin{array}{l}u^{*}_{\varepsilon}(z,t) = - \varepsilon^{-2}B^{T}(t) K(t,\varepsilon)z,\\
v^{*}_{\varepsilon}(z,t) = G^{-1}(t)C^{T}(t)K(t,\varepsilon)z.\end{array}
\end{equation}
By virtue of the results of \cite{Basar-Olsder,Glizer-Kelis-book}, we have the assertion.
\begin{proposition}\label{solvability-K}Let the assumption A1 be fulfilled. Then, the pair $\big(u^{*}_{\varepsilon}(z,t) , v^{*}_{\varepsilon}(z,t)\big)$ is the saddle point of the CCDG. The value of this game has the form 
\begin{equation}\label{J^*_eps} J_{\varepsilon}^{*} = J\big(u^{*}_{\varepsilon}(z,t) , v^{*}_{\varepsilon}(z,t)\big) = z_{0}^{T}K(0,\varepsilon)z_{0}.\end{equation} 
For any $u(z,t) \in {\mathcal E}_{u}\big(v^{*}_{\varepsilon}(z,t)\big)$ and any $v(z,t) \in {\mathcal E}_{v}\big(u^{*}_{\varepsilon}(z,t)\big)$, the following inequality is valid: 
\begin{equation}\label{saddle-point-ineq}J\big(u^{*}_{\varepsilon}(z,t),v(z,t)\big) \le J_{\varepsilon}^{*} \le J\big(u(z,t),v^{*}_{\varepsilon}(z,t)\big).\end{equation}
\end{proposition}

\section{Asymptotic Solution of the Terminal-Value Problem (\ref{eq-K})}

We derive the asymptotic solutions to the problem (\ref{eq-K}) subject to the condition
\begin{equation}\label{lambda_k}\begin{array}{r}\lambda_{p}(t) > 0,\  \ p = 1,...,m_{1},\  \ 1 \le m_{1} < m,\  \ t \in [0,t_{f}],\\
\lambda_{r}(t) \equiv 0,\  \ r = m_{1} + 1,...,m,\  \ t \in [0,t_{f}],\end{array}
\end{equation}
where $\lambda_{k}(t)$, $(k = 1,2,...,m)$ are the entries of the matrix $D_{2}(t)$ (see Remark \ref{D_2-form}, the equation (\ref{form-D_2})).
\subsection{Transformation of the Problem (\ref{eq-K})}

Due to the form of the matrix $S_{u}(\varepsilon)$ (see the equation (\ref{S_u-S_v})), the right-hand side of the differential equation in (\ref{eq-K}) has the singularity for $\varepsilon = 0$. To remove this singularity, we transform the problem (\ref{eq-K}) in a proper way. Namely, using the symmetry of the solution to (\ref{eq-K}) and the form of the matrices $D_{2}(t)$ and $S_{u}(\varepsilon)$ (see the equations (\ref{form-D_2}), (\ref{S_u-S_v}) and (\ref{lambda_k})), we represent this solution as follows:
\begin{equation}\label{block-form}K(t,\varepsilon) = \left[\begin{array}{l}\  \ K_{1}(t,\varepsilon)\ \ \varepsilon K_{2}(t,\varepsilon)\ \  \  \ \varepsilon K_{3}(t,\varepsilon)\\
\varepsilon K_{2}^{T}(t,\varepsilon)\ \ \varepsilon K_{4}(t,\varepsilon)\ \  \  \ \varepsilon^{2} K_{5}(t,\varepsilon)\\
\varepsilon K_{3}^{T}(t,\varepsilon)\ \ \varepsilon^{2} K_{5}^{T}(t,\varepsilon)\ \ \varepsilon^{2} K_{6}(t,\varepsilon)\end{array}\right],
\end{equation}
where the matrices $K_{1}(t,\varepsilon)$, $K_{4}(t,\varepsilon)$ and $K_{6}(t,\varepsilon)$ are of the dimensions $n\times n$, $m_{1}\times m_{1}$ and $(m-m_{1})\times(m-m_{1})$, respectively, and these matrices are symmetric.

Let us partition the matrices $A(t)$ and $S_{v}(t)$ into blocks as follows:
\begin{equation}\label{block-form-A-C}\begin{array}{c}
A(t) = \left[\begin{array}{l}\bar{A}_{1}(t)\ \ \bar{A}_{2}(t)\ \ \bar{A}_{3}(t)\\
\bar{A}_{4}(t)\ \ \bar{A}_{5}(t)\ \ \bar{A}_{6}(t)\\
\bar{A}_{7}(t)\ \ \bar{A}_{8}(t)\ \ \bar{A}_{9}(t)\end{array}\right],\\
S_{v}(t) =  \left[\begin{array}{l}S_{v,1}(t)\ \ S_{v,2}(t)\ \ S_{v,3}(t)\\
S_{v,2}^{T}(t)\ \ S_{v,4}(t)\ \ S_{v,5}(t)\\
S_{v,3}^{T}(t)\ \ S_{v,5}^{T}(t)\ \ S_{v,6}(t)\end{array}\right],\end{array}
\end{equation}
where $\bar{A}_{1}(t) = A_{1}(t)$, $\big[\bar{A}_{2}(t),\bar{A}_{3}(t)\big] = A_{2}(t)$, $\left[\begin{array}{c}\bar{A}_{4}(t)\\ \bar{A}_{7}(t)\end{array}\right] = A_{3}(t)$, $\left[\begin{array}{l}\bar{A}_{5}(t)\ \ \bar{A}_{6}(t)\\ \bar{A}_{8}(t)\ \ \bar{A}_{9}(t)\end{array}\right] = A_{4}(t)$;
the blocks $\bar{A}_{1}(t)$, $\bar{A}_{5}(t)$, $\bar{A}_{9}(t)$ and $S_{v,1}(t)$, $S_{v,4}(t)$, $S_{v,6}(t)$ are of the same dimensions as the corresponding blocks in (\ref{block-form}); $S_{v,j}^{T}(t) = S_{v,j}(t)$, $(j = 1,4,6)$.

In what follows, the asymptotic solution of the problem (\ref{eq-K}) and, therefore, the asymptotic solution of the CCDG are based on the assumption\\
{\bf A2.} $\bar{A}_{3}(t) \equiv 0$, $t \in [0,t_{f}]$.\\
In Section \ref{reduced-game}, the meaning of the assumption A2 is explained. In Section \ref{example}, it is shown that this assumption is fulfilled, for instance, in a real-life pursuit-evasion game.

Using the equations (\ref{new-matrices}),(\ref{S_u-S_v}),(\ref{lambda_k}),(\ref{block-form}),(\ref{block-form-A-C}), and the assumption A2, we can rewrite the terminal-value problem (\ref{eq-K}) in the following equivalent form in the interval $[0,t_{f}]$ (in this problem, for simplicity, we omit the designation of the dependence of the
unknown matrices on $\varepsilon$):
\begin{eqnarray}\label{equiv-eq-K1}dK_{1}(t)/dt = - K_{1}(t)\bar{A}_{1}(t) - \bar{A}_{1}^{T}(t)K_{1}(t)\nonumber\\ +  K_{2}(t)K_{2}^{T}(t) + K_{3}(t)K_{3}^{T}(t) - K_{1}(t)S_{v,1}(t)K_{1}(t)\nonumber\\ - D_{1}(t) - \varepsilon{\mathcal G}_{1}\big(K_{1}(t),K_{2}(t),K_{3}(t),t,\varepsilon\big),\nonumber\\
\end{eqnarray}
\begin{eqnarray}\label{equiv-eq-K2}
\varepsilon dK_{2}(t)/dt = - K_{1}(t)\bar{A}_{2}(t) + K_{2}(t)K_{4}(t)\nonumber\\
- \varepsilon{\mathcal G}_{2}\big(K_{1}(t),K_{2}(t),K_{3}(t),K_{4}(t),K_{5}(t),t,\varepsilon\big),\nonumber\\
\end{eqnarray}
\begin{eqnarray}\label{equiv-eq-K3}
dK_{3}(t)/dt = - K_{2}(t)\bar{A}_{6}(t) - K_{3}(t)\bar{A}_{9}(t) - \bar{A}_{1}^{T}(t)K_{3}(t)\nonumber\\ + K_{2}(t)K_{5}(t) +  K_{3}(t)K_{6}(t)
- K_{1}(t)S_{v,1}(t)K_{3}(t)\nonumber\\ - \varepsilon{\mathcal G}_{3}\big(K_{1}(t),K_{2}(t),K_{3}(t), K_{5}(t), K_{6}(t),t,\varepsilon\big),\nonumber\\
\end{eqnarray}
\begin{eqnarray}\label{equiv-eq-K4}
\varepsilon dK_{4}(t)/dt = K_{4}^{2}(t)
- \Lambda(t)\nonumber\\ - \varepsilon{\mathcal G}_{4}\big(K_{2}(t),K_{4}(t),K_{5}(t),t,\varepsilon\big),\nonumber\\
\end{eqnarray}
\begin{eqnarray}\label{equiv-eq-K5}
\varepsilon dK_{5}(t)/dt = - K_{4}(t)\bar{A}_{6}(t) - \bar{A}_{2}^{T}(t)K_{3}(t) + K_{4}(t)K_{5}(t)\nonumber\\
- \varepsilon{\mathcal G}_{5}\big(K_{2}(t),K_{3}(t),K_{4}(t),K_{5}(t),K_{6}(t),t,\varepsilon\big),
\nonumber\\
\end{eqnarray}
\begin{eqnarray}\label{equiv-eq-K6}
dK_{6}(t)/dt = - K_{5}^{T}(t)\bar{A}_{6}(t) - K_{6}(t)\bar{A}_{9}(t)
- \bar{A}_{6}^{T}(t)K_{5}(t)\nonumber\\ - \bar{A}_{9}^{T}(t)K_{6}(t) + K_{5}^{T}(t)K_{5}(t)
+ K_{6}^{2}(t)\nonumber\\ - K_{3}^{T}(t)S_{v,1}(t)K_{3}(t)
- \varepsilon{\mathcal G}_{6}\big(K_{3}(t),K_{5}(t),K_{6}(t),t,\varepsilon\big),\nonumber\\
\end{eqnarray}
\begin{equation}\label{termin-cond}
K_{1}(t_{f}) = F_{1},\  \  \ K_{\alpha}(t_{f}) = 0,\  \  \ \alpha = 2,3,...,6,
\end{equation}
where ${\mathcal G}_{\alpha}(\cdot)$, $(\alpha = 1,...,6)$ are known smooth functions;
\begin{equation}\label{Lambda}\Lambda(t) \stackrel{\triangle}{=} {\rm diag}\big(\lambda_{1}(t),...,\lambda_{m_{1}}(t)\big).\end{equation}
\begin{remark}\label{difference-representation}The representation (\ref{block-form}) and the system  (\ref{equiv-eq-K1})--(\ref{equiv-eq-K6}) differ considerably from the results of \cite{Glizer-Kelis-book,Kokotovic-IEEE}. These representation and system are essentially novel results in the asymptotic analysis of parameter-dependent matrix Riccati differential equations with singularities, arising in differential games, optimal control problems and $H_{\infty}$ control problems.
\end{remark}

\subsection{Zero-Order Asymptotic Solution of the Terminal-Value Problem (\ref{equiv-eq-K1})--(\ref{equiv-eq-K6}), (\ref{termin-cond})}

We derive the asymptotic solution of the problem (\ref{equiv-eq-K1})--(\ref{equiv-eq-K6}), (\ref{termin-cond}) subject to the following assumption:\\
{\bf A3.} The matrix-valued functions $A_{i}(t)$, $(i = 1,...,4)$, $C_{j}(t)$, $D_{j}(t)$, $(j = 1,2)$, $G(t)$ are continuously differentiable in the interval $[0,t_{f}]$.

The problem (\ref{equiv-eq-K1})--(\ref{equiv-eq-K6}), (\ref{termin-cond}) is a singularly perturbed terminal-value problem where $K_{1}(t)$, $K_{3}(t)$, $K_{6}(t)$ are slow states, while $K_{2}(t)$, $K_{4}(t)$, $K_{5}(t)$ are fast states. Based on the Boundary Functions Method (see, e.g., \cite{Vasil'eva-SIAM}), we look for the zero-order asymptotic solution of this problem in the form
\begin{equation}\label{asymp-K}\begin{array}{c}K_{\alpha,0}(t,\varepsilon) = K_{\alpha,0}^{o}(t) + K_{\alpha,0}^{b}(\tau),\\
\alpha = 1,...,6,\  \  \  \ \tau = (t-t_{f})/\varepsilon,\end{array}
\end{equation}
where the terms with the superscript $o$ constitute the so-call outer solution, the terms with the superscript $b$ are the boundary corrections in the left-hand neighborhood of $t = t_{f}$; $\tau \le 0$ is a new independent variable, called the stretched time; for any $t \in [0,t_{f})$, $\tau \rightarrow - \infty$ as $\varepsilon \rightarrow + 0$. Equations and boundary conditions for the asymptotic solution terms are obtained substituting $ K_{\alpha,0}(t,\varepsilon)$, $(\alpha = 1,...,6)$ into the problem (\ref{equiv-eq-K1})--(\ref{equiv-eq-K6}), (\ref{termin-cond}) instead of $K_{\alpha}(t)$ and equating the coefficients for the same power of $\varepsilon$ on both sides of the resulting equations, separately depending on $t$ and on $\tau$. Moreover, the boundary corrections should satisfy the condition $\lim_{\tau \rightarrow - \infty}K_{\alpha,0}^{b}(\tau) = 0$, $(\alpha = 1,...,6)$.

Boundary correction terms $K_{\alpha_{s},0}^{b}(\tau)$, $(\alpha_{s} = 1,3,6)$ satisfy the problems
\begin{equation}\label{eq-K^b_js}dK_{\alpha_{s},0}^{b}(\tau)/d\tau = 0,\ \lim_{\tau \rightarrow - \infty} K_{\alpha_{s},0}^{b}(\tau) = 0,\ \alpha_{s} = 1,3,6,\end{equation}
yielding
\begin{equation}\label{mathcal-K^b_js}
K_{\alpha_{s},0}^{b}(\tau) \equiv 0,\  \  \ \tau \le 0,\  \  \ \alpha_{s} = 1,3,6.
\end{equation}

For the outer solution terms, we obtain the following equations and conditions:
\begin{eqnarray}\label{eq-outer-terms-K1}dK_{1,0}^{o}(t)/dt = - K_{1,0}^{o}(t)\bar{A}_{1}(t) - \bar{A}_{1}^{T}(t)K_{1,0}^{o}(t)\nonumber\\ + K_{2,0}^{o}(t)\big(K_{2,0}^{o}(t)\big)^{T} + K_{3,0}^{o}(t)\big( K_{3,0}^{o}(t)\big)^{T}\nonumber\\ - K_{1,0}^{o}(t)S_{v,1}(t)K_{1,0}^{o}(t) - D_{1}(t),\ K_{1,0}^{o}(t_{f}) = F_{1},\nonumber\\
\end{eqnarray}
\begin{eqnarray}\label{eq-outer-terms-K2}
0 = - K_{1,0}^{o}(t)\bar{A}_{2}(t) + K_{2,0}^{o}(t)K_{4,0}^{o}(t),
\end{eqnarray}
\begin{eqnarray}\label{eq-outer-terms-K3}
dK_{3,0}^{o}(t)/dt = - K_{2,0}^{o}(t)\bar{A}_{6}(t) - K_{3,0}^{o}(t)\bar{A}_{9}(t)\nonumber\\ - \bar{A}_{1}^{T}(t)K_{3,0}^{o}(t) + K_{2,0}^{o}(t)K_{5,0}^{o}(t) + K_{3,0}^{o}(t)K_{6,0}^{o}(t)\nonumber\\
- K_{1,0}^{o}(t)S_{v,1}(t)K_{3,0}^{o}(t),\ K_{3,0}^{o}(t_{f}) = 0,\nonumber\\
\end{eqnarray}
\begin{eqnarray}\label{eq-outer-terms-K4}
0 = K_{4}^{2}(t) - \Lambda(t),
\end{eqnarray}
\begin{eqnarray}\label{eq-outer-terms-K5}
0 = - K_{4,0}^{o}(t)\bar{A}_{6}(t) - \bar{A}_{2}^{T}(t)K_{3,0}^{o}(t) + K_{4,0}^{o}(t)K_{5,0}^{o}(t),
\nonumber\\
\end{eqnarray}
\begin{eqnarray}\label{eq-outer-terms-K6}
dK_{6,0}^{o}(t)/dt = - \big(K_{5,0}^{o}(t)\big)^{T}\bar{A}_{6}(t) - K_{6,0}^{o}(t)\bar{A}_{9}(t)\nonumber\\
- \bar{A}_{6}^{T}(t)K_{5,0}^{o}(t) - \bar{A}_{9}^{T}(t)K_{6,0}^{o}(t) + \big(K_{5,0}^{o}(t)\big)^{T}K_{5,0}^{o}(t)\nonumber\\
+ \big(K_{6,0}^{o}(t)\big)^{2} - \big(K_{3,0}^{o}(t)\big)^{T}S_{v,1}(t)K_{3,0}^{o}(t),\ K_{6,0}^{o}(t_{f}) = 0.\nonumber\\
\end{eqnarray}

Resolving the algebraic equations (\ref{eq-outer-terms-K2}), (\ref{eq-outer-terms-K4}, (\ref{eq-outer-terms-K5}) with respect to $K_{2,0}^{o}(t)$, $ K_{4,0}^{o}(t)$, $K_{5,0}^{o}(t)$, $t \in [0,t_{f}]$, we obtain
\begin{equation}\label{outer-solution}\begin{array}{r}K_{4,0}^{o}(t) = \Lambda^{1/2}(t) = {\rm diag}\big(\sqrt{\lambda_{1}(t)},...,\sqrt{\lambda_{m_{1}}(t)}\big),\\
K_{2,0}^{o}(t) = K_{1,0}^{o}(t)\bar{A}_{2}(t)\Lambda^{-1/2}(t),\\
K_{5,0}^{o}(t) = \bar{A}_{6}(t) + \Lambda^{-1/2}(t)\bar{A}_{2}^{T}(t)K_{3,0}^{o}(t).\end{array}
\end{equation}

Substituting the expressions for $K_{2,0}^{o}(t)$ and $K_{5,0}^{o}(t)$ into the differential equation (\ref{eq-outer-terms-K3}) for $K_{3,0}^{o}(t)$, we obtain after a routine algebra
\begin{equation}\label{eq-K_30^o}\begin{array}{r}dK_{3,0}^{o}(t)/dt = - K_{3,0}^{o}(t)\bar{A}_{9}(t) -  \bar{A}_{1}^{T}(t)K_{3,0}^{o}(t)\\ + K_{1,0}^{o}(t)\bar{A}_{2}(t)\Lambda^{-1}(t)\bar{A}_{2}^{T}(t)K_{3,0}^{o}(t)
+ K_{3,0}^{o}(t)K_{6,0}^{o}(t)\\ - K_{1,0}^{o}(t)S_{v,1}(t)K_{3,0}^{o}(t),\  \ K_{3,0}^{o}(t_{f}) = 0,\end{array}
\end{equation}
which yields
\begin{equation}\label{K_30^o}K_{3,0}^{o}(t) \equiv 0,\  \  \ t \in [0,t_{f}].\end{equation}

Using (\ref{outer-solution}) and (\ref{K_30^o}), we can rewrite the terminal-value problems (\ref{eq-outer-terms-K1}) and (\ref{eq-outer-terms-K6}) for $K_{1,0}^{o}(t)$ and $K_{6,0}^{o}(t)$, respectively, as:
\begin{equation}\label{eq-K_10^o}\begin{array}{r}dK_{1,0}^{o}(t)/dt = - K_{1,0}^{o}(t)\bar{A}_{1}(t) -  \bar{A}_{1}^{T}(t)K_{1,0}^{o}(t)\\ + K_{1,0}^{o}(t)\big[\bar{A}_{2}(t)\Lambda^{-1}(t)\bar{A}_{2}^{T}(t) - S_{v,1}(t)\big]K_{1,0}^{o}(t)\\
- D_{1}(t),\  \  \ t \in [0,t_{f}],\  \ K_{1,0}^{o}(t_{f}) = F_{1},\end{array}
\end{equation}
\begin{equation}\label{eq-K_60^o}\begin{array}{r}dK_{6,0}^{o}(t)/dt = - K_{6,0}^{o}(t)\bar{A}_{9}(t) -  \bar{A}_{9}^{T}(t)K_{6,0}^{o}(t)\\ + \big(K_{6,0}^{o}(t)\big)^{2}
- \bar{A}_{6}^{T}(t)\bar{A}_{6}(t),\ t \in [0,t_{f}],\ K_{6,0}^{o}(t_{f}) = 0.\end{array}
\end{equation}
The equation (\ref{eq-K_10^o}) represents the terminal-value problem for the game-theoretic matrix Riccati differential equation, while the equation (\ref{eq-K_60^o}) represents the terminal-value problem for the control-theoretic matrix Riccati differential equation.

By virtue of the results of \cite{Kwakernaak-Sivan}, the terminal-value problem (\ref{eq-K_60^o}) has the solution $K_{6,0}^{o}(t)$ in the entire interval $[0,t_{f}]$. However, the terminal-value problem (\ref{eq-K_10^o}) may not have the solution in the entire interval $[0,t_{f}]$. Therefore, in what follows, we assume\\
{\bf A4.} The terminal-value problem (\ref{eq-K_10^o}) has the solution $K_{1,0}^{o}(t)$ in the entire interval $[0,t_{f}]$.

\begin{remark}\label{cond-exist-solution} Sufficient conditions for the existence of the solution to the terminal-value problem (\ref{eq-K_10^o}), mentioned in the assumption A4, can be found in \cite{Glizer-Turetsky-Optim} and references therein.
\end{remark}

Proceed to the derivation of the boundary correction terms $K_{\alpha_{f},0}^{b}(\tau)$, $(\alpha_{f} = 2,4,5)$. Using the equations (\ref{mathcal-K^b_js}),(\ref{outer-solution}),(\ref{K_30^o}),(\ref{eq-K_10^o}), we obtain the terminal-value problem for these terms
\begin{eqnarray}\label{eq-K_20^b=K_40^b-K_50^b}dK_{2,0}^{b}(\tau)/d\tau = K_{2,0}^{b}(\tau)\Lambda^{1/2}(t_{f})\nonumber\\ + F_{1}\bar{A}_{2}(t_{f})\Lambda^{-1/2}(t_{f})K_{4,0}^{b}(\tau)
+ K_{2,0}^{b}(\tau)K_{4,0}^{b}(\tau),\nonumber\\ \tau \le 0,\  \  \ K_{2,0}^{b}(0) = - F_{1}\bar{A}_{2}(t_{f})\Lambda^{-1/2}(t_{f}),\nonumber\\
dK_{4,0}^{b}(\tau)/d\tau = K_{4,0}^{b}(\tau)\Lambda^{1/2}(t_{f}) + \Lambda^{1/2}(t_{f})K_{4,0}^{b}(\tau)\nonumber\\
+ \big(K_{4,0}^{b}(\tau)\big)^{2},\  \  \ \tau \le 0,\  \ K_{4,0}^{b}(0) = - \Lambda^{1/2}(t_{f}),\nonumber\\
dK_{5,0}^{b}(\tau)/d\tau = \Lambda^{1/2}(t_{f})K_{5,0}^{b}(\tau) + K_{4,0}^{b}(\tau)K_{5,0}^{b}(\tau),\nonumber\\
\tau \le 0,\  \  \ K_{5,0}^{b}(0) = - \bar{A}_{6}(t_{f}).\nonumber
\end{eqnarray}
This problem yields the solution
\begin{equation}\label{K_20^b=K_40^b-K_50^b}\begin{array}{r}K_{2,0}^{b}(\tau) = -2F_{1}\bar{A}_{2}(t_{f})\Lambda^{-1/2}(t_{f}){\mathcal H}_{1}^{2}(\tau)
{\mathcal H}_{2}(\tau),\\
K_{4,0}^{b}(\tau) = - 2\Lambda^{1/2}(t_{f}){\mathcal H}_{1}^{2}(\tau){\mathcal H}_{2}(\tau),\\
K_{5,0}^{b}(\tau) = - 2{\mathcal H}_{1}(\tau){\mathcal H}_{2}(\tau)\bar{A}_{6}(t_{f}),\end{array}
\end{equation}
where ${\mathcal H}_{1}(\tau) \stackrel{\triangle}{=} \exp\big(\Lambda^{1/2}(t_{f})\tau\big)$, ${\mathcal H}_{2}(\tau) \stackrel{\triangle}{=} \big[I_{m_{1}} + {\mathcal H}_{1}^{2}(\tau)\big]^{-1}$.\\
Due to (\ref{lambda_k}) and (\ref{Lambda}), $K_{\alpha_{f},0}^{b}(\tau)$, $(\alpha_{f} = 2,4,5)$ are exponentially decaying for $\tau \le 0$, i.e.,
\begin{eqnarray}\label{ineq-K_20^b=K_40^b-K_50^b}\big\|K_{2,0}^{b}(\tau)\big\| \le a_{2}\exp(2\beta\tau),\ \ \big\|K_{4,0}^{b}(\tau)\big\| \le a_{4}\exp(2\beta\tau),\nonumber\\
\big\|K_{5,0}^{b}(\tau)\big\| \le a_{5}\exp(\beta\tau),\nonumber
\end{eqnarray}
where $a_{2} > 0$, $a_{4} > 0$, $a_{5} > 0$ are some constants; $\beta \stackrel{\triangle}{=} \min_{p \in \{1,...,m_{1}\}}\sqrt{\lambda_{p}(t_{f})} > 0$.

\subsection{Reduced Differential Game}\label{reduced-game}

Let us set $\varepsilon = 0$ in the differential game, consisting of the singularly perturbed dynamics (\ref{eq-x}),(\ref{sing-pert-eq-y}) and the non-cheap control cost functional (\ref{no-cheap-contr-funct}). Re-denoting the notations $x$, $y$, $\widehat{u}$, $v$ and ${\widehat J}$ in the resulting system and functional with $x_{\rm r}$, $y_{\rm r}$, $\widehat{u}_{\rm r}$, $v_{\rm r}$ and ${\widehat J}_{\rm r}$, respectively, we obtain
\begin{equation}\label{reduced-eq-x}\begin{array}{c}dx_{\rm r}(t)/dt = A_{1}(t)x_{\rm r}(t) + A_{2}(t)y_{\rm r}(t)\\
+ C_{1}(t)v_{\rm r}(t),\  \  \ t \in [0,t_{f}],\  \  \ x_{\rm r}(0) = x_{0},\end{array}
\end{equation}
\begin{equation}\label{reduced-eq-y}
\widehat{u}_{\rm r}(t) = 0,\  \  \  \ t \in [0,t_{f}],
\end{equation}
\begin{equation}\label{reduced-funct}\begin{array}{c}\widehat{J}_{\rm r} = x_{\rm r}^{T}(t_{f})F_{1}x_{\rm r}(t_{f}) +  \int_{0}^{t_{f}}\big[x_{\rm r}^{T}(t)D_{1}(t)x_{\rm r}(t)\\
+ y_{\rm r}^{T}(t)D_{2}(t)y_{\rm r}(t) + \widehat{u}_{\rm r}^{T}(t)\widehat{u}_{\rm r}(t) - v_{\rm r}^{T}(t)G(t)v_{\rm r}(t)\big]dt.\end{array}
\end{equation}
Due to (\ref{reduced-eq-y}), the functional (\ref{reduced-funct}) becomes
\begin{equation}\label{reduced-funct-2}\begin{array}{c}\widehat{J}_{\rm r} = x_{\rm r}^{T}(t_{f})F_{1}x_{\rm r}(t_{f}) +  \int_{0}^{t_{f}}\big[x_{\rm r}^{T}(t)D_{1}(t)x_{\rm r}(t)\\
+ y_{\rm r}^{T}(t)D_{2}(t)y_{\rm r}(t) - v_{\rm r}^{T}(t)G(t)v_{\rm r}(t)\big]dt.\end{array}
\end{equation}
In the set (\ref{reduced-eq-x}),(\ref{reduced-funct-2}), the variable $y_{\rm r}(t)$, $t \in [0,t_{f}]$ does not satisfy any equation. Hence, we can choose
it to satisfy a desirable property of the system (\ref{reduced-eq-x}) and the functional (\ref{reduced-funct-2}),
i.e., as a control in these system and functional. Moreover, the control of the maximizer $v_{\rm r}(t)$ is present in (\ref{reduced-eq-x}),(\ref{reduced-funct-2}), while a minimizer's control does not appear in these system and functional. Therefore, it is reasonable to choose $y_{\rm r}(t)$ as a minimizer's control. This observation means that the functional (\ref{reduced-funct-2}) depends on $y_{\rm r}$ and $v_{\rm r}$, i.e., $\widehat{J}_{\rm r} = \widehat{J}_{\rm r}(y_{\rm r},v_{\rm r})$, where $y_{\rm r}(\cdot)$ is the control of the minimizing player, while $v_{\rm r}(\cdot)$ is the control of the maximizing player.
In this game, both players are aware of all the game’s data and of the game’s current position $\{x_{\rm r}(t), t\}$.\\
Since the matrix $D_{2}(t)$ is not invertible (see Remark \ref{D_2-form} and the equation (\ref{lambda_k})), then the game (\ref{reduced-eq-x}),(\ref{reduced-funct-2}) can be directly solved neither by the Isaacs’s MinMax principle nor by the Bellman-Isaacs equation method, i.e., it is singular. However, if the assumption A2 is fulfilled, the game (\ref{reduced-eq-x}),(\ref{reduced-funct-2}) becomes as:
\begin{equation}\label{reduced-eq-x-new}\begin{array}{c}dx_{\rm r}(t)/dt = A_{1}(t)x_{\rm r}(t) + \bar{A}_{2}(t)y_{\rm r,1}(t)\\
+ C_{1}(t)v_{\rm r}(t),\  \  \ t \in [0,t_{f}],\  \  \ x_{\rm r}(0) = x_{0},\end{array}
\end{equation}
\begin{equation}\label{reduced-funct-new}\begin{array}{c}\widehat{J}_{\rm r}(y_{\rm r,1},v_{\rm r}) = x_{\rm r}^{T}(t_{f})F_{1}x_{\rm r}(t_{f}) + \int_{0}^{t_{f}}\big[x_{\rm r}^{T}(t)D_{1}(t)x_{\rm r}(t)\\
+ y_{\rm r,1}^{T}(t)\Lambda(t)y_{\rm r,1}(t) - v_{\rm r}^{T}(t)G(t)v_{\rm r}(t)\big]dt,\end{array}
\end{equation}
where $y_{\rm r,1}(t) \in E^{m_{1}}$ is the upper block of the control vector $y_{\rm r}(t)$ and it is the minimizer's control in the game (\ref{reduced-eq-x-new})-(\ref{reduced-funct-new}).
Since the matrix $\Lambda(t)$ is positive definite for all $t \in [0,t_{f}]$, the game (\ref{reduced-eq-x-new})-(\ref{reduced-funct-new}) is regular, i.e., the Isaacs’s MinMax principle and the Bellman-Isaacs equation method can be directly used for its solution.
We call this game the Reduced Differential Game (RDG). Thus, {\bf due to the assumption A2, the game  (\ref{reduced-eq-x}),(\ref{reduced-funct-2}) becomes the regular RDG}.\\
For all $(x_{\rm r},t) \in E^{n}\times[0,t_{f}]$, consider the functions
\begin{equation}\label{y_r1^*-v_r^*}\begin{array}{c}y_{\rm r,1}^{*}(x_{\rm r},t) = - \Lambda^{-1}(t)\bar{A}_{2}(t)K_{1,0}^{o}(t)x_{\rm r}\\ v_{\rm r}^{*}(x_{\rm r},t) = G^{-1}(t)C_{1}(t) K_{1,0}^{o}(t)x_{\rm r},\end{array}\end{equation}
where $K_{1,0}^{o}(t)$ is the solution of the problem (\ref{eq-K_10^o}).
Taking into account that $S_{v,1}(t) = C_{1}(t)G^{-1}(t)C_{1}^{T}$ and using the results of \cite{Basar-Olsder,Glizer-Kelis-book}, we obtain the assertion.
\begin{proposition}\label{solvability-RDG}Let the assumption A4 be fulfilled. Then, the pair $\big(y_{\rm r,1}^{*}(x_{\rm r},t),v_{\rm r}^{*}(x_{\rm r},t)\big)$ is the saddle point of the RDG. The value of this game has the form 
\begin{equation}\label{game-value-reduced}\widehat{J}_{\rm r}^{*} = \widehat{J}_{\rm r}\big(y_{\rm r,1}^{*}(x_{\rm r},t),v_{\rm r}^{*}(x_{\rm r},t)\big) = x_{0}^{T}K_{1,0}^{o}(0)x_{0}.
\end{equation}
\end{proposition}

\subsection{Justification of the Asymptotic Solution to the Problem (\ref{equiv-eq-K1})--(\ref{equiv-eq-K6}), (\ref{termin-cond})}

By a straightforward generalization of the results of \cite{Glizer-Turetsky-Symmetry} to the terminal-value problem (\ref{equiv-eq-K1})--(\ref{equiv-eq-K6}), (\ref{termin-cond})), we obtain the assertion.
\begin{lemma}\label{justification-asymptotics-K_1-6}
Let the assumptions A2-A4 be fulfilled. Then, there exists a number $\varepsilon_{0}>0$ such that, for all $\varepsilon\in(0,\varepsilon_{0}]$, the problem (\ref{equiv-eq-K1})--(\ref{equiv-eq-K6}), (\ref{termin-cond}) has the unique solution $\{K_{\alpha}(t,\varepsilon),\ (\alpha = 1,...,6)\}$, in the entire interval $[0,t_f]$.
This solution satisfies the inequalities $\big\|K_{\alpha}(t,\varepsilon) - K_{\alpha,0}(t,\varepsilon)\big\| \leq
a\varepsilon$, $t \in [0,t_f]$, where $K_{\alpha,0}(t,\varepsilon)$, $(\alpha = 1,...,6)$, are given by (\ref{asymp-K}); $a > 0$ is some constant independent of $\varepsilon$.
\end{lemma}
Lemma \ref{justification-asymptotics-K_1-6} directly yields the assertion.
\begin{corollary}\label{saddle-point-existence}Let the assumptions A2-A4 be fulfilled. Then, for any $\varepsilon\in(0,\varepsilon_{0}]$, the assumption A1 is fulfilled and $K(t,\varepsilon)$ has the form (\ref{block-form}). Furthermore, all the statements of Proposition \ref{solvability-K} are valid.
\end{corollary}

\section{Approximate-Saddle Point of the CCDG}
For any $\varepsilon \in (0,\varepsilon_{0}]$, consider the matrix-valued function
\begin{equation}\label{block-form-K_0}K_{0}(t,\varepsilon) = \left[\begin{array}{l}\  \ K_{1,0}(t,\varepsilon)\ \ \varepsilon K_{2,0}(t,\varepsilon)\ \  \  \ \varepsilon K_{3,0}(t,\varepsilon)\\
\varepsilon K_{2,0}^{T}(t,\varepsilon)\ \ \varepsilon K_{4,0}(t,\varepsilon)\ \  \  \ \varepsilon^{2} K_{5,0}(t,\varepsilon)\\
\varepsilon K_{3,0}^{T}(t,\varepsilon)\ \ \varepsilon^{2} K_{5,0}^{T}(t,\varepsilon)\ \ \ \varepsilon^{2} K_{6,0}(t,\varepsilon)\end{array}\right].
\end{equation}
In (\ref{u*-v*}), let us replace $K(t,\varepsilon)$ with $K_{0}(t,\varepsilon)$. Thus, we obtain the controls of the minimizer and the maximizer in the CCDG
\begin{equation}\label{u_0-v_0}\begin{array}{l}u_{\varepsilon 0}(z,t) = - \varepsilon^{-2}B^{T}(t) K_{0}(t,\varepsilon)z,\\
v_{\varepsilon 0}(z,t) = G^{-1}(t)C^{T}(t)K_{0}(t,\varepsilon)z.\end{array}
\end{equation}
Since $u_{\varepsilon 0}(z,t)$ and $v_{\varepsilon 0}(z,t)$ are linear with respect to $z$ for all $(t,\varepsilon) \in [0,t_{f}]\times(0,\varepsilon_{0}]$, then $\big(u_{\varepsilon 0}(z,t),v_{\varepsilon 0}(z,t)\big) \in UV$.
Employing this pair of controls in the CCDG and using (\ref{new-matrices}),(\ref{S_u-S_v}), yields $J_{\varepsilon 0} \stackrel{\triangle}{=} J\big(u_{\varepsilon 0}(z,t),v_{\varepsilon 0}(z,t)\big)$ as:
\begin{equation}\label{output}\begin{array}{l}J_{\varepsilon 0} = z_{\varepsilon 0}^{T}(t_{f})Fz_{\varepsilon 0}(t_{f})
+ \int_{0}^{t_{f}}z_{\varepsilon 0}^{T}(t){\mathcal D}(t,\varepsilon)z_{\varepsilon 0}(t)dt,\\
{\mathcal D}(t,\varepsilon) \stackrel{\triangle}{=} D(t) + K_{0}(t,\varepsilon)[S_{u}(\varepsilon) - S_{v}(t)]K_{0}(t,\varepsilon),\end{array}
\end{equation}
where $z_{\varepsilon 0}(t)$ satisfies of the initial-value problem
\begin{equation}\label{init-value-pr}\begin{array}{c}dz_{\varepsilon 0}(t)/dt = {\mathcal A}(t,\varepsilon)z_{\varepsilon 0}(t),\ t \in [0,t_{f}],\ z_{\varepsilon 0}(t) = z_{0},\\
{\mathcal A}(t,\varepsilon) \stackrel{\triangle}{=} A(t) - [S_{u}(\varepsilon) - S_{v}(t)]K_{0}(t,\varepsilon).\end{array}
\end{equation}
Similarly to the work \cite{Glizer-Turetsky-Symmetry}, the value $J_{\varepsilon 0}$ can be expressed as:
\begin{equation}\label{output-new-form}J_{\varepsilon 0} = z_{0}^{T}{\mathcal L}(0,\varepsilon)z_{0},\end{equation}
where ${\mathcal L}(t,\varepsilon)$ satisfies the terminal-value problem
\begin{equation}\label{eq-mathcal-L}\begin{array}{l}d{\mathcal L}(t,\varepsilon)/dt = - {\mathcal L}(t,\varepsilon){\mathcal A}(t,\varepsilon) - {\mathcal A}^{T}(t,\varepsilon){\mathcal L}(t,\varepsilon)\\ - {\mathcal D}(t,\varepsilon),\  \  \  \
t \in [0,t_{f}],\  \  \ {\mathcal L}(t_{f},\varepsilon) = F.\end{array}
\end{equation}
Using the symmetry of the matrix ${\mathcal L}(t,\varepsilon)$, we represent it in the block form similar to (\ref{block-form})
\begin{equation}\label{block-form-mathcal-L}{\mathcal L}(t,\varepsilon) = \left[\begin{array}{l}\  \ {\mathcal L}_{1}(t,\varepsilon)\ \ \varepsilon {\mathcal L}_{2}(t,\varepsilon)\ \  \  \ \varepsilon {\mathcal L}_{3}(t,\varepsilon)\\
\varepsilon {\mathcal L}_{2}^{T}(t,\varepsilon)\ \ \varepsilon {\mathcal L}_{4}(t,\varepsilon)\ \  \  \ \varepsilon^{2} {\mathcal L}_{5}(t,\varepsilon)\\
\varepsilon {\mathcal L}_{3}^{T}(t,\varepsilon)\ \ \varepsilon^{2} {\mathcal L}_{5}^{T}(t,\varepsilon)\ \ \varepsilon^{2} {\mathcal L}_{6}(t,\varepsilon)\end{array}\right].
\end{equation}
\begin{lemma}\label{estim-K-L}Let the assumptions A2-A4 be fulfilled. Then, there exists a positive number $\varepsilon_{1}\le \varepsilon_{0}$ such that, for all $\varepsilon\in(0,\varepsilon_{1}]$, the following inequalities are valid: $\big\|K_{\alpha}(t,\varepsilon) - {\mathcal L}_{\alpha}(t,\varepsilon)\big\| \le c\varepsilon^{2}$, $t \in [0,t_{f}]$, $(\alpha = 1,...,6)$, where $c > 0$ is some constant independent of $\varepsilon$.
\end{lemma}
\begin{pf} Let us introduce the matrix-valued function
\begin{equation}\label{Delta-K_L}\Delta K_{\mathcal L}(t,\varepsilon) \stackrel{\triangle}{=} K(t,\varepsilon) - {\mathcal L}(t,\varepsilon),\ t \in [0,t_{f}],\ \varepsilon \in (0,\varepsilon_{0}].\end{equation}
Using (\ref{eq-K}),(\ref{eq-mathcal-L}) and the notation $\Delta K_{0}(t,\varepsilon) \stackrel{\triangle}{=} K(t,\varepsilon) - K_{0}(t,\varepsilon)$, we obtain that, for any $\varepsilon \in (0,\varepsilon_{0}]$,
$\Delta K_{\mathcal L}(t,\varepsilon)$ satisfies the terminal-value problem in the interval $[0,t_{f}]$
$$\begin{array}{l}d\Delta K_{\mathcal L}(t,\varepsilon)/dt = - \Delta K_{\mathcal L}(t,\varepsilon){\mathcal A}(t,\varepsilon) - {\mathcal A}^{T}(t,\varepsilon)\Delta K_{\mathcal L}(t,\varepsilon)\\
+ \Delta K_{0}(t,\varepsilon)[S_{u}(\varepsilon) - S_{v}(t)]\Delta K_{0}(t,\varepsilon),\  \ \Delta K_{\mathcal L}(t,\varepsilon) = 0.\end{array}$$
By virtue of the results of \cite{A_F_I_J}, we have the solution of this problem as:
\begin{equation}\label{solution-Delta-K_L}\begin{array}{l}\Delta K_{\mathcal L}(t,\varepsilon) = \int_{t_{f}}^{t}\Phi^{T}(\sigma,t,\varepsilon)\Delta K_{0}(\sigma,\varepsilon)\big[S_{u}(\varepsilon)\\ - S_{v}(\sigma)\big]\Delta K_{0}(\sigma,\varepsilon)\Phi(\sigma,t,\varepsilon)d\sigma,\end{array}
\end{equation}
where for any $t \in [0,t_{f}]$ and $\varepsilon \in (0,\varepsilon_{0}]$, the $(n+m)\times(n+m)$-matrix-valued function $\Phi(\sigma,t,\varepsilon)$ is the solution of the following problem in the interval $\sigma \in [t,t_{f}]$:
$$d\Phi(\sigma,t,\varepsilon)/d\sigma = {\mathcal A}(\sigma,\varepsilon)\Phi(\sigma,t,\varepsilon),\ \ \Phi(t,t,\varepsilon) = I_{n+m}.$$

To estimate $\Phi(\sigma,t,\varepsilon)$, let us partition it into blocks as:
\begin{equation}\label{Phi-block-form}\Phi(\sigma,t,\varepsilon) = \left[\begin{array}{l}\Phi_{1}(\sigma,t,\varepsilon)\ \ \Phi_{2}(\sigma,t,\varepsilon)\ \ \Phi_{3}(\sigma,t,\varepsilon)\\
\Phi_{4}(\sigma,t,\varepsilon)\ \ \Phi_{5}(\sigma,t,\varepsilon)\ \ \Phi_{6}(\sigma,t,\varepsilon)\\
\Phi_{7}(\sigma,t,\varepsilon)\ \ \Phi_{8}(\sigma,t,\varepsilon)\ \ \Phi_{9}(\sigma,t,\varepsilon)\end{array}\right],
\end{equation}
where the blocks are of the same dimensions as the corresponding blocks in (\ref{block-form}).\\
Using the block form of the matrix ${\mathcal A}(t,\varepsilon)$ (see the equations (\ref{init-value-pr}) and (\ref{S_u-S_v}),(\ref{block-form-A-C}),(\ref{block-form-K_0})), as well as
the equations (\ref{lambda_k}),(\ref{mathcal-K^b_js}),(\ref{outer-solution}),(\ref{K_30^o}),(\ref{Phi-block-form}) and the assumption A2, we obtain the following estimates for all $0 \le t \le \sigma \le t_{f}$:
\begin{equation}\label{estimates-Phi-blocks}\begin{array}{l}\big\|\Phi_{\kappa_{1}}(\sigma,t,\varepsilon)\big\| \le c_{\Phi},\ \ \kappa_{1} = 1,4,7,9,\  \ \varepsilon \in (0,\varepsilon_{1}],\\
\big\|\Phi_{\kappa_{2}}(\sigma,t,\varepsilon)\big\| \le c_{\Phi}\varepsilon,\ \ \kappa_{2} = 2,3,6,8\  \ \varepsilon \in (0,\varepsilon_{1}],\\
\big\|\Phi_{5}(\sigma,t,\varepsilon)\big\| \le c_{\Phi}\big[\varepsilon + \exp\big(\bar{\beta}(t - \sigma)/\varepsilon\big)\big],\ \varepsilon \in (0,\varepsilon_{1}],\end{array}
\end{equation}
where $0 < \varepsilon_{1} \le \varepsilon_{0}$ is some sufficiently small number; $c_{\Phi} > 0$ is some constant independent of $\varepsilon$; $\bar{\beta} = \min_{t \in [0,t_{f}]}\min_{p\in \{1,...,m_{1}\}}\sqrt{\lambda_{p}(t)} > 0$.\\
Now, using the equations (\ref{Delta-K_L}),(\ref{solution-Delta-K_L}), the block forms of the matrices $S_{u}(\varepsilon)$, $S_{v}(t)$, $K(t,\varepsilon)$, $K_{0}(t,\varepsilon)$, ${\mathcal L}(t,\varepsilon)$, $\Phi(t,\varepsilon)$ (see the equations (\ref{S_u-S_v}),(\ref{block-form-A-C}),(\ref{block-form}),(\ref{block-form-K_0}),(\ref{block-form-mathcal-L}),(\ref{Phi-block-form})), as well as Lemma \ref{justification-asymptotics-K_1-6} and the estimates (\ref{estimates-Phi-blocks}), we obtain by a routine algebra the validity of the inequalities stated in the lemma.
\end{pf}
\begin{theorem}\label{J_eps*-J_eps0}Let the assumptions A2-A4 be fulfilled. Then, for all $\varepsilon \in (0,\varepsilon_{1}]$, the following inequality is valid: $\big|J_{\varepsilon}^{*} - J_{\varepsilon 0}\big| \le c\varepsilon^{2}\psi(z_{0},\varepsilon)$, where $\psi(z_{0},\varepsilon) \stackrel{\triangle}{=} \|z_{0,1}\|^{2} + \varepsilon(2\|z_{0,1}\|\|z_{0,2}\| + \|z_{0,2}\|^{2} + 2\|z_{0,1}\|\|z_{0,3}\|) + \varepsilon^{2}(2\|z_{0,2}\|\|z_{0,3}\|\\ + \|z_{0,3}\|^{2}),$ $z_{0,1} \in E^{n}$, $z_{0,2} \in E^{m_{1}}$, $z_{0,3} \in E^{m-m_{1}}$, ${\rm col}(z_{0,1},z_{0,2},z_{0,3}) = z_{0}$.
\end{theorem}
\begin{pf} The statement of the theorem directly follows from Proposition \ref{solvability-K}, the equations (\ref{block-form}),(\ref{output-new-form}),(\ref{block-form-mathcal-L}) and Lemma \ref{estim-K-L}.
\end{pf}
Now, let us study asymptotic behaviour of $J_{u}\big(u_{\varepsilon 0}(z,t);z_{0}\big)$ and $J_{v}\big(v_{\varepsilon 0}(z,t);z_{0}\big)$. Using Definitions \ref{guaran-res-u}, \ref{guaran-res-v}, the equation (\ref{u_0-v_0}) and the results of \cite{Glizer-Turetsky-Optim}, we obtain
\begin{equation}\label{guar-res-expresssions}\begin{array}{l}J_{u,\varepsilon0} \stackrel{\triangle}{=}J_{u}\big(u_{\varepsilon 0}(z,t);z_{0}\big) = z_{0}^{T}{\mathcal M}(0,\varepsilon)z_{0},\ \ \varepsilon \in (0,\varepsilon_{0}],\\
J_{v,\varepsilon0} \stackrel{\triangle}{=}J_{v}\big(v_{\varepsilon 0}(z,t);z_{0}\big) = z_{0}^{T}{\mathcal N}(0,\varepsilon)z_{0},\ \  \varepsilon \in (0,\varepsilon_{0}],\end{array}
\end{equation}
where  the matrix-valued functions ${\mathcal M}(t,\varepsilon)$, ${\mathcal N}(t,\varepsilon)$ satisfy the terminal-value problems in the interval $[0,t_{f}]$
\begin{equation}\label{eq-mathcal-M}\begin{array}{l}d{\mathcal M}(t,\varepsilon)/dt = - {\mathcal M}(t,\varepsilon){\mathcal A}_{u}(t,\varepsilon) - {\mathcal A}_{u}^{T}(t,\varepsilon){\mathcal M}(t,\varepsilon)\\ - {\mathcal M}(t,\varepsilon)S_{v}(t){\mathcal M}(t,\varepsilon) - {\mathcal D}_{u}(t,\varepsilon),\  \ {\mathcal M}(t_{f},\varepsilon) = F,\end{array}
\end{equation}
\begin{equation}\label{eq-mathcal-N}\begin{array}{l}d{\mathcal N}(t,\varepsilon)/dt = - {\mathcal N}(t,\varepsilon){\mathcal A}_{v}(t,\varepsilon) - {\mathcal A}_{v}^{T}(t,\varepsilon){\mathcal N}(t,\varepsilon)\\ + {\mathcal N}(t,\varepsilon)S_{u}(\varepsilon){\mathcal N}(t,\varepsilon) - {\mathcal D}_{v}(t,\varepsilon),\  \ {\mathcal N}(t_{f},\varepsilon) = F.\end{array}
\end{equation}
In (\ref{eq-mathcal-M})-(\ref{eq-mathcal-N}): ${\mathcal A}_{u}(t,\varepsilon) \stackrel{\triangle}{=} A(t) - S_{u}(\varepsilon)K_{0}(t,\varepsilon)$,\\ ${\mathcal D}_{u}(t,\varepsilon) \stackrel{\triangle}{=} D(t) + K_{0}(t,\varepsilon)S_{u}(\varepsilon)K_{0}(t,\varepsilon)$, ${\mathcal A}_{v}(t,\varepsilon) \stackrel{\triangle}{=} A(t)\\ + S_{v}(t)K_{0}(t,\varepsilon)$, ${\mathcal D}_{v}(t,\varepsilon) \stackrel{\triangle}{=} D(t) - K_{0}(t,\varepsilon)S_{v}(t)K_{0}(t,\varepsilon)$.\\
Based on (\ref{guar-res-expresssions}),(\ref{eq-mathcal-M}),(\ref{eq-mathcal-N}), we obtain (similarly to Lemma \ref{estim-K-L} and Theorem \ref{J_eps*-J_eps0}) the following assertion:
\begin{theorem}\label{J_eps*-J_u-J_v}Let the assumptions A2-A4 be fulfilled. Then, there exists a positive number $\varepsilon_{2} \le \varepsilon_{0}$ such that, for all $\varepsilon \in (0,\varepsilon_{2}]$, the following inequalities are valid: $\big|J_{\varepsilon}^{*} - J_{u,\varepsilon 0}\big| \le c_{u}\varepsilon^{2}\psi_{u}(z_{0},\varepsilon)$ and $\big|J_{\varepsilon}^{*} - J_{v,\varepsilon 0}\big| \le c_{v}\varepsilon^{2}\psi_{v}(z_{0},\varepsilon)$,
where $c_{u} > 0$ and $c_{v} > 0$ are some constants independent of $\varepsilon$; $\psi_{u}(z_{0},\varepsilon) \stackrel{\triangle}{=} \|z_{0,1}\|^{2} + \varepsilon(2\|z_{0,1}\|\|z_{0,2}\| + \|z_{0,2}\|^{2} + 2\|z_{0,1}\|\|z_{0,3}\|) + \varepsilon^{2}(2\|z_{0,2}\|\|z_{0,3}\| + \|z_{0,3}\|^{2})$; $\psi_{v}(z_{0},\varepsilon) \stackrel{\triangle}{=} \|z_{0,1}\|^{2} + \varepsilon(2\|z_{0,1}\|\|z_{0,2}\| + 2\|z_{0,1}\|\|z_{0,3}\|) + \varepsilon^{2}(\|z_{0,2}\| + \|z_{0,3}\|)^{2}$.
\end{theorem}
As a direct consequence of Theorems \ref{J_eps*-J_eps0} and \ref{J_eps*-J_u-J_v}, we have:
\begin{corollary}\label{approx-ineq}Let the assumptions A2-A4 be fulfilled. Then, for any $u(z,t) \in {\mathcal E}_{u}\big(v_{\varepsilon 0}(z,t))$, any $v(z,t) \in {\mathcal E}_{v}\big(u_{\varepsilon 0}(z,t))$ and all $\varepsilon \in \big(0,\min\{\varepsilon_{1},\varepsilon_{2}\}]$, the following inequality is valid: $J\big(u_{\varepsilon 0}(z,t),v(z,t)\big) - \varepsilon^{2}\big[c\psi(z_{0},\varepsilon) + c_{u}\psi_{u}(z_{0},\varepsilon)\big] \le J_{\varepsilon 0} \le J\big(u(z,t),v_{\varepsilon 0}(z,t)\big) + \varepsilon^{2}\big[c\psi(z_{0},\varepsilon) + c_{v}\psi_{v}(z_{0},\varepsilon)\big].$
\end{corollary}
\begin{remark}\label{approx-saddle-point}Due to the comparison of Theorem \ref{J_eps*-J_eps0} and Corollary \ref{approx-ineq} with Proposition \ref{solvability-K}, we can call the pair $\big(u_{\varepsilon 0}(z,t),v_{\varepsilon 0}(z,t)\big)$ and the inequality of Corollary \ref{approx-ineq} an approximate-saddle point and an approximate-saddle point inequality, respectively, in the CCDG.
\end{remark}
\section{Example}\label{example}
We consider the pursuit-evasion engagement between two flying vehicles modeled by the following system:
\begin{equation}\label{syst-example}\begin{array}{l}dx(t)/dt = y_{1}(t) + v_{1}(t),\  \  \  \  \  \ \  \  \  \  \  \  \ \ x(0) = 0,\\
dy_{1}(t)/dt = y_{2}(t) + u_{1}(t) + v_{2}(t),\  \ y_{1}(0) = 2,\\
dy_{2}(t)/dt = - y_{2}(t) + u_{2}(t),\  \  \  \  \ \  \  \  \ \ y_{2}(0) = 1,\end{array}
\end{equation}
where $t \in [0,1.5]$

The system (\ref{syst-example}), being a modification of the results of \cite{Shinar-Glizer-Turetsky-Annals}, represents a linearized kinematic model of a planar engagement between two flying vehicles (the players): a pursuer and an evader. The pursuer is controlled by the lateral acceleration command $u_{2}(t)$ and the additional lateral acceleration $u_{1}(t)$, while the evader is controlled by the lateral acceleration $v_{2}(t)$ and the lateral velocity $v_{1}(t)$. The state coordinates $x(t)$ and $y_{1}(t)$ are the lateral separation and the relative lateral velocity of the vehicles, while $y_{2}(t)$ is the main lateral acceleration of the pursuer generated by the control $u_{2}(t)$. The behaviour of the players is evaluated by the cost functional
\begin{equation}\label{functional-example}\begin{array}{l}J(u,v) = 0.5x^{2}(t_{f}) + \int_{0}^{1.5}\big[6.4x^{2}(t) + 10y_{1}^{2}(t)\\
+ \varepsilon^{2}\big(u_{1}^{2}(t) + u_{2}^{2}(t)\big) - 5v_{1}^{2}(t) - 4v_{2}^{2}(t)\big]dt,\end{array}\end{equation}
where $u = {\rm col}(u_{1},u_{2})$, $v = {\rm col}(v_{1},v_{2})$; $\varepsilon > 0$ is a small parameter. The cost functional (\ref{functional-example}) is minimized by the pursuer and maximized by the evader.
Assuming that both players of the game (\ref{syst-example})-(\ref{functional-example}) are aware of the game's data and of the game's current position $\big(x(t),y_{1}(t),y_{2}(t),t\big)$, we obtain that this game is a particular case of the CCDG (\ref{eq-x})-(\ref{perf-ind}) where $n = 1$, $m = 2$, $l = 2$, $m_{1} = 1$, $F_{1} = 0.5$, $D_{1}(t) \equiv 6.4$, $D_{2}(t) \equiv {\rm diag}(10,0)$, $G(t) \equiv {\rm diag}(5 , 4)$, $t_{f} = 1.5$, $A_{1}(t) \equiv 0$, $A_{2}(t) \equiv [1 , 0]$, $A_{3}(t) \equiv {\rm col}(0 , 0)$,
$$\begin{array}{l}
A_{4}(t) \equiv \left[\begin{array}{l}0\  \  \ \ \ \ 1\\0\  \ -1\end{array}\right],\ C_{1}(t) \equiv [1 , 0],\ C_{2}(t) \equiv \left[\begin{array}{l}0\  \ 1\\ 0\  \ 0\end{array}\right].\end{array}$$
Using these data, as well as Remark \ref{D_2-form} and the equations (\ref{new-matrices}),(\ref{S_u-S_v}),(\ref{block-form-A-C}),(\ref{Lambda}), we obtain in this example
\begin{equation}\label{bar-A-S_valpha}\begin{array}{l}\bar{A}_{1}(t) = \bar{A}_{3}(t) = \bar{A}_{4}(t) = \bar{A}_{5}(t) = \bar{A}_{7}(t) = \bar{A}_{8}(t) \equiv 0,\\
\bar{A}_{2}(t) = \bar{A}_{6}(t) \equiv 1,\ \bar{A}_{9}(t) \equiv -1,\ S_{v,1}(t) \equiv 0.2,\\ \Lambda(t) \equiv 10.\end{array}
\end{equation}
Due to (\ref{asymp-K}),(\ref{mathcal-K^b_js}),(\ref{outer-solution}),(\ref{K_30^o})-(\ref{K_20^b=K_40^b-K_50^b}),(\ref{functional-example}),(\ref{bar-A-S_valpha}), we obtain
$$\begin{array}{l}K_{1,0}(t,\varepsilon) = K_{1,0}^{o}(t)
 = 8\tan\big(\arctan(1/16) + 1.2 - 0.8t\big),\\
K_{2,0}(t,\varepsilon) = \big[K_{1,0}^{o}(t) + \exp(2\beta\tau)\big(1 + \exp(2\beta\tau)\big)^{-1}\big]/\beta,\\
 K_{3,0}(t,\varepsilon) \equiv 0,\\
K_{4,0}(t,\varepsilon) = \beta\big[1 - 2\exp(2\beta\tau)\big(1 + \exp(2\beta\tau)\big)^{-1}\big],\\
K_{5,0}(t,\varepsilon) = 1 - 2\exp(\beta\tau)\big(1 + \exp(2\beta\tau)\big)^{-1},\\
K_{6,0}(t,\varepsilon) = K_{6,0}^{o}(t) =\gamma\tanh(\gamma(t-1.5))\times\\
(\gamma\tanh(\gamma(t-1.5))-2)^{-1},
\end{array}$$
where $t \in [0,1.5]$, $\tau = (t - 1.5)/\varepsilon$, $\beta = \sqrt{10}$, $\gamma = \sqrt{2}$.\\
Using $K_{\alpha,0}(t,\varepsilon)$, $(\alpha = 1,...,6)$ and the equations (\ref{block-form-K_0}),(\ref{u_0-v_0}), we derive the components of the approximate-saddle point in the game (\ref{syst-example})-(\ref{functional-example}) as:
$$\begin{array}{l}u_{\varepsilon 0}(z,t) = - \left[\begin{array}{l}\big(K_{2,0}(t,\varepsilon)x +K_{4, 0}(t,\varepsilon)y_{1}\big)/\varepsilon\\ + K_{5,0}(t,\varepsilon)y_{2},\\
K_{5,0}(t,\varepsilon)y_{1} + K_{6,0}^{o}(t)y_{2}\end{array}\right],\\
v_{\varepsilon0}(z,t) = \left[\begin{array}{l}0.2\big(K_{1,0}^{o}(t)x+\varepsilon K_{2,0}(t,\varepsilon)y_{1}\big),\\
0.25\varepsilon\big(K_{2,0}(t,\varepsilon)x+K_{4,0}(t,\varepsilon)y_{1}\\ +
\varepsilon K_{5,0}(t,\varepsilon)y_{2}\big)\end{array}\right],\end{array}
$$
where $z = {\rm col}(x,y_{1},y_{2})$.

\begin{figure}[ht]
\includegraphics[scale=0.135]{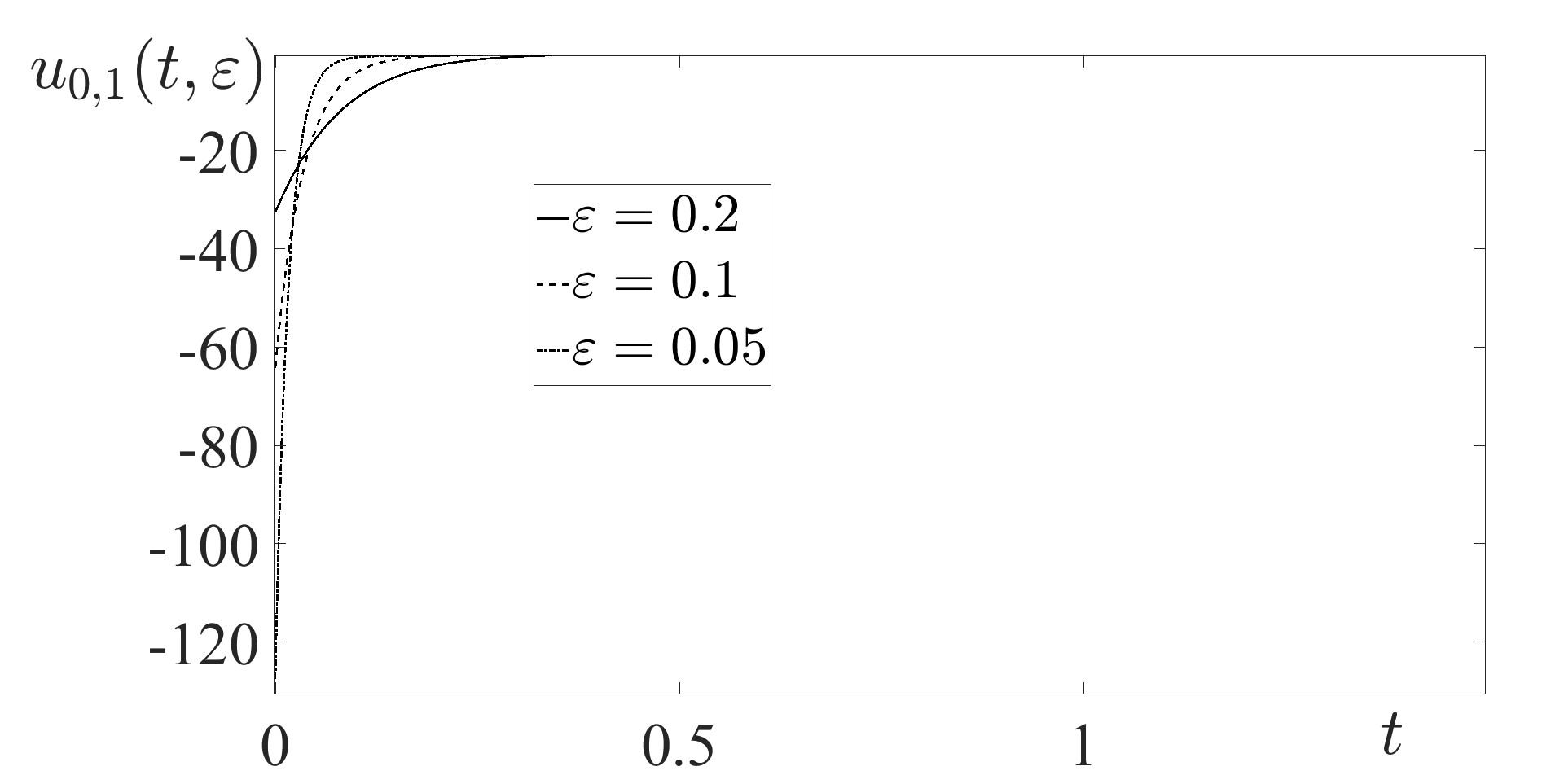}
\caption{Time history of  $u_{0,1}(t,\varepsilon)$}
\label{ueps0_1-fig}
\end{figure}

\begin{figure}[ht]
\includegraphics[scale=0.135]{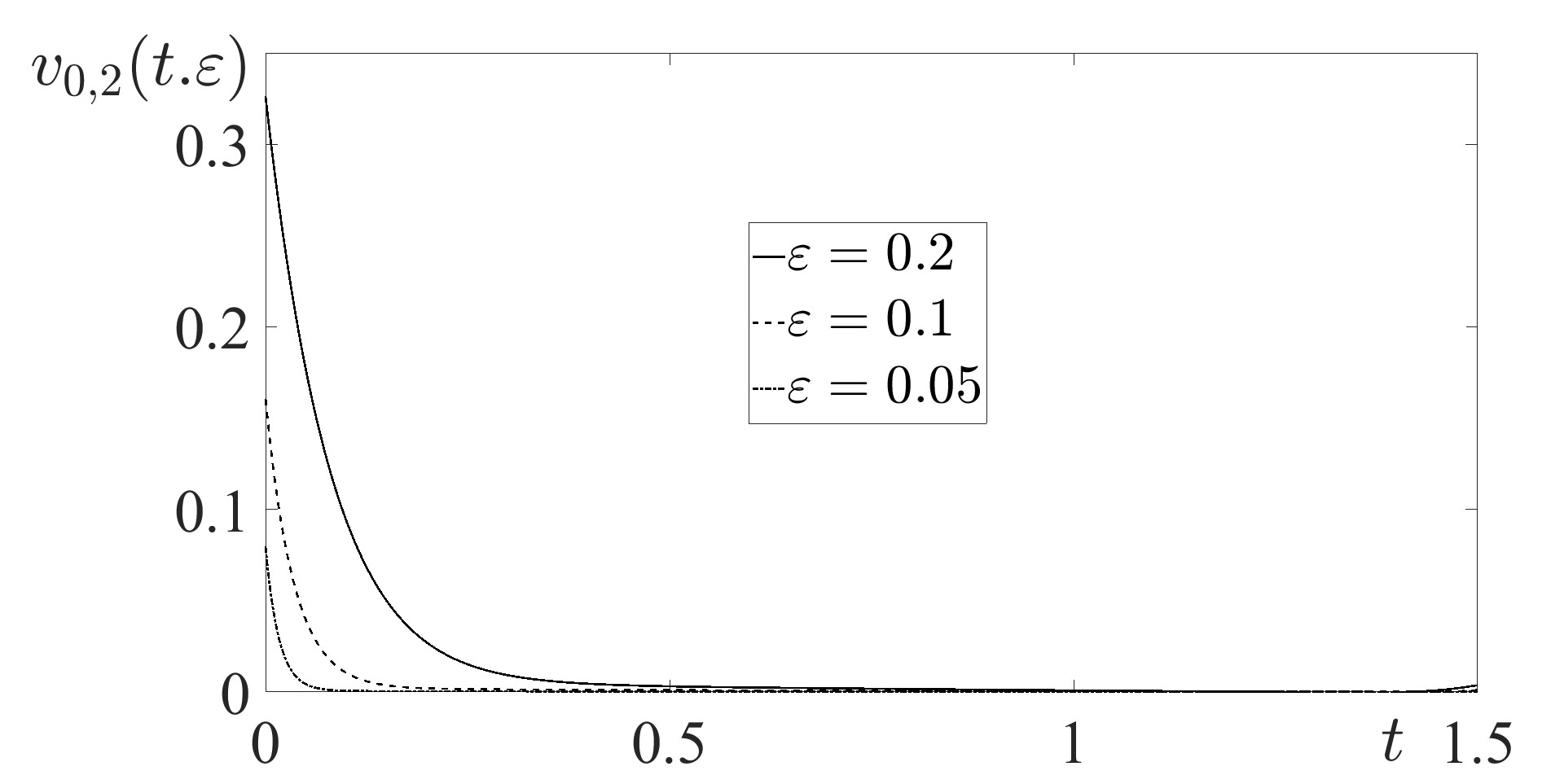}
\caption{Time history of $v_{0,2}(t,\varepsilon)$}
\label{veps0_2-fig}
\end{figure}

Due to the form of the first (upper) entry of $u_{\varepsilon 0}(z,t)$, its time realization $u_{0,1}(t,\varepsilon)$ along the trajectory of (\ref{syst-example}), generated by $\big(u_{\varepsilon 0}(z,t),v_{\varepsilon 0}(z,t)\big)$, is expected to have an impulse-like behaviour for $\varepsilon \rightarrow + 0$. The time realizations of the second entry of $u_{\varepsilon 0}(z,t)$ and of both entries of $v_{\varepsilon 0}(z,t)$ do not have such behaviour. Moreover, the time realization $v_{0,2}(t,\varepsilon)$ of the second entry of $v_{\varepsilon 0}(z,t)$ is expected to tend to zero for $\varepsilon \rightarrow + 0$. These features are illustrated by Figs. \ref{ueps0_1-fig}  and \ref{veps0_2-fig}, respectively, where the time histories of $u_{0,1}(t,\varepsilon)$ and $v_{0,2}(t,\varepsilon)$ are depicted for decreasing values of $\varepsilon > 0$.\\
In Table \ref{appr_SP_errors-table}, the absolute errors $\Delta J_{\varepsilon 0}=|J_\varepsilon^*-J_{\varepsilon 0}|$ and the relative errors $\delta J_{\varepsilon 0} = \left(\Delta J_{\varepsilon 0}/J_\varepsilon^*\right)\cdot 100\%$ are presented for $\varepsilon=0.2, 0.1, 0.05$. The game value $J_\varepsilon^*$ was calculated by the numerical solution of the Riccati equation (\ref{eq-K}). It is seen that the values $J_{\varepsilon 0}$ provide a reliable approximation of the game value: the absolute error decreases with $\varepsilon$, the relative errors are small.

\begin{table}[ht]
\begin{tabular}{|c|c|c|c|c|}
\hline
$\varepsilon$ & $J_\varepsilon^*$ & $J_{\varepsilon 0}$ & $\Delta J_{\varepsilon 0}$ &$\delta J_{\varepsilon 0}$, \%\\
\hline
$0.2$ & $3.1892$ & $3.2247$ & $0.0355$ &  $1.11$\\
\hline
$0.1$ & $1.4184$ & $1.4258$ & $0.0074$ & $0.52$\\
\hline
$0.05$ & $ 0.6696$ & $0.6733$  & $0.0037$  & $0.55$\\
\hline
\end{tabular}

\caption{Errors of $J_{\varepsilon 0}$\  \  \  \  \  \  \ \  \  \  \  \  \  \ \  \  \  \  \  \  \  \  \  \  \ }
\label{appr_SP_errors-table}
\end{table}

In Table \ref{appr_Ju_errors-table}, the absolute errors $\Delta J_{u,\varepsilon 0}=|J_\varepsilon^*-J_{u,\varepsilon0}|$ and the relative errors $\delta J_{u,\varepsilon 0}=\left(\Delta J_{u,\varepsilon 0}/J_\varepsilon^*\right)\cdot 100\%$ are presented for $\varepsilon=0.2, 0.1, 0.05$. The respective absolute errors $\Delta J_{v,\varepsilon 0}=|J_\varepsilon^*-J_{v,\varepsilon 0}|$ and the relative errors $\delta_{v,\varepsilon 0}=\left(\Delta J_{v,\varepsilon 0}/J_\varepsilon^*\right)\cdot 100\%$ are presented in Table \ref{appr_Jv_errors-table}. The values $J_{u,\varepsilon 0}$ and $J_{v,\varepsilon 0}$ were calculated by the numerical solution of the Riccati equations (\ref{eq-mathcal-M}) and (\ref{eq-mathcal-N}), respectively.
It is seen that the values $J_{u,\varepsilon 0}$ and $J_{v,\varepsilon 0}$ provide a reliable approximation of the game value. Moreover, as it can be seen from Tables \ref{appr_SP_errors-table} - \ref{appr_Jv_errors-table}, the approximation of the game value by $J_{u,\varepsilon 0}$ and $J_{v,\varepsilon 0}$ is (as a rule) more accurate than such an approximation by $J_{\varepsilon 0}$. It is reasonable, because in $J_{u,\varepsilon 0}$ and $J_{v,\varepsilon 0}$ only one control's approximation (either $u_{\varepsilon 0}(z,t)$ or $v_{\varepsilon 0}(z,t)$) is used, while in $J_{\varepsilon 0}$ both controls' approximations are used. From the other hand, $J_{\varepsilon 0}$ requires a simpler calculation than $J_{u,\varepsilon 0}$ and $J_{v,\varepsilon 0}$, because the problem (\ref{eq-mathcal-L}) is linear, while the problems (\ref{eq-mathcal-M}) and (\ref{eq-mathcal-N}) are non-linear.

\begin{table}[ht]
\begin{tabular}{|c|c|c|c|c|}
\hline
$\varepsilon$ & $J_\varepsilon^*$ & $J_{u,\varepsilon 0}$ & $\Delta J_{u,\varepsilon 0}$ &$\delta J_{u,\varepsilon 0}$, \%\\
\hline
$0.2$ & $3.1892$ & $3.2401$ & $0.051$ &  $1.6$\\
\hline
$0.1$ & $1.4184$ & $1.4234$ & $0.005$ & $0.35$\\
\hline
$0.05$ & $ 0.6696$ & $0.6702$  & $5.57\cdot 10^{-4}$  & $0.08$\\
\hline
\end{tabular}

\caption{Errors of $J_{u,\varepsilon 0}$\  \  \  \  \  \  \  \  \  \  \  \  \  \  \  \  \  \  }
\label{appr_Ju_errors-table}
\end{table}

\begin{table}[ht]
\begin{tabular}{|c|c|c|c|c|}
\hline
$\varepsilon$ & $J_\varepsilon^*$ & $J_{v,\varepsilon 0}$ & $\Delta J_{v,\varepsilon 0}$ &$\delta J_{v,\varepsilon 0}$, \%\\
\hline
$0.2$ & $3.1892$ & $3.1755$ & $0.0137$ &  $0.43$\\
\hline
$0.1$ & $1.4184$ & $1.4176$ & $7.89\cdot 10^{-4}$ & $0.06$\\
\hline
$0.05$ & $ 0.6696$ & $0.6696$  & $4.73\cdot 10^{-5}$  & $0.007$\\
\hline
\end{tabular}

\caption{Errors of $J_{v,\varepsilon 0}$\  \  \  \  \  \  \ \  \  \  \  \  \  \  \  \  \  \  }
\label{appr_Jv_errors-table}
\end{table}

\section{Conclusions}
In this paper, a two-player finite-horizon zero-sum linear-quadratic differential game was studied in the case where the control cost of the minimizing player (the minimizer) is much smaller than the control cost of its opponent and the state cost. This smallness is due to the presence of the small multiplier $\varepsilon > 0$ in the control cost of the minimizer. Thus, the considered game is a cheap control game. One more feature of the game is that its dynamics has slow and fast state variables, and the cost of the fast state in the integrand of the cost functional is a positive semi-definite (but non-zero) quadratic form. The asymptotic analysis of the game with respect to $\varepsilon$ was carried out. This analysis includes the asymptotic solution of the corresponding singularly perturbed game-theoretic matrix Riccati differential equation, derivation of the approximate-saddle point and obtaining the approximation of the game value. It should be noted that the derivation of the asymptotic solution to the Riccati equation requires the developing the considerably novel method. Based on the theoretical results, the approximate solution of the real-life pursuit-evasion game was derived. Numerical simulation shows that a very accurate approximation of the game's solution is obtained even for not too small values of $\varepsilon > 0$.

\end{document}